\documentclass[12 pt]{article}%
\usepackage{amsmath, amsfonts, amsthm, color,latexsym}
\usepackage{amsmath, ulem}
\usepackage{amsfonts}
\usepackage{amssymb}
\usepackage{color, soul}
\usepackage[all]{xy}
\usepackage{graphicx}%
\providecommand{\U}[1]{\protect\rule{.1in}{.1in}}
\allowdisplaybreaks[4]
\newtheorem{theorem}{Theorem}[section]

\newtheorem{corollary}[theorem]{Corollary}
\newtheorem{example}[theorem]{Example}

\newtheorem{lemma}[theorem]{Lemma}
\newtheorem{final remark}[theorem]{Final Remark}

\allowdisplaybreaks[4]

\newcommand {\N} {\mathbb{N}}

\begin{document}

\title{On the duality of $DW$-compact operators and $DW$-$DP$ operators}
\author{Geraldo Botelho\thanks{Supported by FAPEMIG grants RED-00133-21 and APQ-01853-23}~ and Ariel Mon\c c\~ao\thanks{Supported by a CAPES scholarship. \newline 2020 Mathematics Subject Classification: 46B42, 46B50, 47B07, 47B65.\newline Keywords: Banach lattices, $DW$-compact operators, $DW$-$DP$ operators, Dunford-Pettis operators, almost Dunford-Pettis operators, $AM$-compact operators. }}
\date{}
\maketitle

\begin{abstract} We give a necessary condition and a sufficient condition on the Banach lattices $E$ and $F$ so that an operator from $E$ to $F$ is $DW$-compact whenever its adjoint is $DW$-compact. We do the same, with different conditions, for $DW$-$DP$ operators. Moreover, we characterize the Banach lattices $E$ and $F$ for which the adjoint of every $DW$-compact operator from $E$ to $F$ is $DW$-compact.
\end{abstract}

\section{Introduction}

Let ${\cal C}$ be a class of bounded linear operators from a Banach lattice $E$ to a Banach lattice $F$. The {\it adjoint problem} for the class ${\cal C}$ asks whether or not the following implications hold:
$$T \colon E \longrightarrow F \mbox{ belongs to } {\cal C} \Longleftrightarrow T^* \colon F^* \longrightarrow E^* \mbox{ belongs to } {\cal C} , $$
where $E^*$ is the topological dual of $E$ and $T^*$ is the adjoint of $T$.

In the realm of operators between Banach space, it is known the adjoint problem has positive answer for several classes, for example, for compact operators (Schauder's Theorem), for weakly compact operators (Gantmacher's Theorem) and for approximable operators (Hutton's Theorem). Sometimes none of the implications hold, for instance for the class of Dunford-Pettis operators (also known as completely continuous operators), which are the operators that send weakly null sequences to norm null sequences.

In Banach lattice theory, the adjoint problem has been studied for several classes of operators, for instance: $AM$-compact operators (\cite[Theorem 125.6]{zaanen}, see also \cite{Aqzzouz4}), order weakly compact operators \cite{aqzzouzhmichane}, almost Dunford-Pettis operators \cite{Aqzzouz3}, positive Dunford-Pettis operators \cite{Aqzzouz4}, weak Dunford-Pettis operators \cite{Aqzzouz8}. In this paper we address the adjoint problem for the class of $DW$-compact operators, introduced in \cite{Jin}; as well as one of the implications of the adjoint problem for the class of $DW$-$DP$ operators, recently introduced in \cite{Jin2}. More precisely, in Section 2 we give a necessary condition and a sufficient condition on the Banach lattices $E$ and $F$ so that an operator from $E$ to $F$ is $DW$-compact whenever its adjoint is $DW$-compact; and we give a necessary and sufficient condition so that the adjoint of every $DW$-compact operator from $E$ to $F$ is $DW$-compact. In Section 3 we give a necessary condition and a sufficient condition so that an operator from $E$ to $F$ is $DW$-$DP$ whenever its adjoint is $DW$-$DP$.

Throughout the paper, $E$ and $F$ shall denote Banach lattices, $X$ will denote a Banach space and $X^*$ denotes its (topological) dual. Operators are always linear and bounded. By ${\rm id}_X$ we mean the identity operator on $X$. For background on Banach lattices and positive/regular operators, we refer to \cite{Ali1, MeyerPeter, Schaefer}. Undefined concepts and unexplained symbols can be found in \cite{Ali1}.

%We shall use repeatedly, sometimes without warning, the following characterizations:

\section{$DW$-compact operators}

Let $A$ be a subset of a Banach lattice $E$. The solid hull of $A$ is the set ${\rm Sol}(A) = \{x \in E : |x| \leq |y| \mbox{ for some } y \in A\}$ \cite[p.\,171]{Ali1}. The set $A$ is said to be {\it disjoint weakly compact} if every disjoint sequence in ${\rm Sol}(A)$ is weakly null (see \cite{Wnuk99}). Relatively weakly compact subsets of a Banach lattice are disjoint weakly compact %Os subconjuntos ordem limitados e os subconjuntos fraco relativamente compactos de um reticulado de Banach são disjuntos fracamente compactos
\cite[Proposition 2.5.12(iii)]{MeyerPeter}.

According to \cite{Jin}, an operator $T \colon E \longrightarrow X$ is {\it $DW$-compact} if $T$ sends disjoint weakly compact subsets of $E$ to relatively compact subsets of $X$.

In order to state a very useful characterization, we recall two extensively studied classes of operators, which were introduced, respectively, in \cite{sanchez} and \cite{Dodds}. An operator $T \colon E \longrightarrow X$ is said to be:\\
$\bullet$ {\it Almost Dunford-Pettis} if $T$ maps disjoint weakly null sequence in $E$ to norm null sequences in $X$.\\
$\bullet$ {\it $AM$-compact} if $T$ maps order bounded subsets of $E$ to relatively compact subsets of $X$.

\begin{theorem}{\rm \cite[Theorem 2.5]{Jin}} \label{teo.2.5.dw} The following are equivalent for an operator $T \colon E \longrightarrow X$.\\
{\rm (a)} $T$ is $DW$-compact.\\
{\rm (b)} $T$ is almost Dunford-Pettis and $AM$-compact.\\
{\rm (c)} $T$ is Dunford-Pettis and $AM$-compact.
\end{theorem}

The characterizations above shall be used henceforth without further notice. Next we shall see that none of the implications of the adjoint problem for $DW$-compact operators holds true in general.

\begin{example}\rm Recall that a Banach space has the Schur property if weakly null sequences are norm null. Let us show that, for a Banach lattice $E$,
$$ E \mbox{ has the Schur property} \Longleftrightarrow {\rm id_E} \mbox{ is a $DW$-compact operator}. $$
Suppose that $E$ has the Schur property. Then, obviously, ${\rm id_E}$ is a Dunford-Pettis operator and $E$ has the positive Schur property. By \cite[Theorem and Lemma 1]{Wnuk2}, $E$ has order continuous norm, thus every Dunford-Pettis operator from $E$ to any Banach space is $AM$-compact by \cite[Theorem 2.4.2(vi)]{MeyerPeter}. So, ${\rm id_E}$ is a $DW$-compact operator. The converse follows immediately from Theorem \ref{teo.2.5.dw}.

Letting $E$ be an infinite dimensional Banach lattice with the Schur property (for instance, $\ell_1$), it is well known that $E^*$ fails the Schur property (see, e.g, \cite[Corollary 11]{mujicaarkiv}). Therefore, ${\rm id}_E$ is a $DW$-compact operator whose adjoint $ ({\rm id}_E)^*= {\rm id}_{E^*}$ fails to be $DW$-compact. Starting with an infinite dimensional Banach lattice $E$ whose dual $E^*$ has the Schur property (for instance, $c_0$), the same argument shows that $({\rm id}_E)^* = {\rm id}_{E^*}$ is $DW$-compact whereas ${\rm id}_E$ is not.
\end{example}

Sometimes the duality of $DW$-compact operators holds by collapsing to the duality of compact operators:

\begin{example}\rm Let $E$ and $F$ be Banach lattices so that $E^*$ and $F^{**}$ have order continuous norms and let $T \colon E \longrightarrow F$ be an operator. Combining Schauder's Theorem \cite[Theorem 7.2.7]{livro} with \cite[Corollary 2.2]{Jin}, and using the obvious fact that compact operators are $DW$-compact, we get
$$T \mbox{ is $DW$-compact} \Longleftrightarrow T \mbox{ is compact} \Longleftrightarrow T^* \mbox{ is compact} \Longleftrightarrow T^* \mbox{ is $DW$-compact}.$$
\end{example}

So, for each of the implications of the adjoint problem for $DW$-compact operators, our job is to find conditions on $E$ and $F$, not more restrictive than $E^*$ and $F^{**}$ having order continuous norms, so that the implication holds.

Our first result establishes a sufficient condition and a necessary condition under which a positive or, equivalently, a regular operator is $DW$-compact whenever its adjoint is $DW$-compact.

\begin{theorem}\label{Teo.3} Consider the following assertions about the Banach lattices $E$ and $F$.
\vspace*{-0.4em}
\begin{enumerate}
\item[\rm(i)] $F$ has the positive Schur property.
\vspace*{-0.4em}
\item[\rm(ii)] A regular operator $T\colon E\longrightarrow F$ is $DW$-compact whenever $T^*$ is $AM$-compact.
\vspace*{-0.4em}
\item[\rm(iii)] A regular operator $T\colon E\longrightarrow F$ is $DW$-compact whenever $T^*$ is $DW$-compact.
\vspace*{-0.4em}
\item[\rm(iv)] A positive operator $T\colon E\longrightarrow F$ is $DW$-compact whenever $T^*$ is $DW$-compact.
\vspace*{-0.4em}
\item[\rm(v)] One of the following alternatives holds:
\begin{enumerate}
\vspace*{-0.5em}
\item[\rm(a)] $F$ has order continuous norm.
\vspace*{-0.3em}
\item[\rm(b)] $E^*$ is discrete.
\vspace*{-0.3em}
\item[\rm(c)] The norm of $E^*$ is not order continuous and $F^{**}$ is not  discrete.
\end{enumerate}
\end{enumerate}
\vspace*{-0.4em}
Then {\rm(i)} $\!\Longrightarrow$ \!{\rm(ii)} $\!\Longrightarrow$ {\!\rm(iii)} $\!\Longrightarrow$ \!{\rm(iv)} $\!\Longrightarrow$ \!{\rm(v)}.
\end{theorem}

\begin{proof} {\rm(i)} $\!\Longrightarrow$ \!{\rm(ii)} As mentioned before, the positive Schur property of $F$ assures that its norm is order continuous. Let $T\colon E\longrightarrow F$ be a regular operator for which $T^*$ is $AM$-compact.  By \cite[Theorem 2.1]{Aqzzouz4}, $T$ is $AM$-compact. The regularity of $T$ and the positive Schur property of $F$ guarantee that $T$ is almost Dunford-Pettis, hence $T$ is $DW$-compact.

\medskip

\noindent {\rm(ii)} $\!\Longrightarrow$ \!{\rm(iii)} $\!\Longrightarrow$ \!{\rm(iv)} are obvious.

\medskip

\noindent{\rm(iv)} $\!\Longrightarrow$ \!{\rm(v)} Assume, for the sake of contradiction, that (a), (b) and (c) do not hold. As (a) and (b) fail, the proof of \cite[Theorem 2.12]{Jin} yields the existence of operators $S,T\colon E\longrightarrow F$ such that $0\leq S\leq T$, $T$ is compact and $S$ is not $AM$-compact. Then $T^*$ is compact, hence $DW$-compact, by Schauder's Theorem and $S$ is not $DW$-compact. Since (c) does not hold, $E^*$ has order continuous norm or $F^{**}$ is discrete. As $T^*$ is compact and $0\leq S^*\leq T^*$, $S^*$ is $DW$-compact by \cite[Theorem 2.12]{Jin}. Condition (iv) gives that $S$ is $DW$-compact, a contradiction that completes the proof.
\end{proof}

Now we turn to the other implication of the adjoint problem for $DW$-compact operators. The following lemma is certainly known to experts, we give a short proof for the benefit of the reader.

\begin{lemma}\label{lemafim}
Let $E$ be a Banach lattice. For every nonzero functional $x^*\in E^*$, there is a positive vector $x\in E$ such that $\|x\|\leq1$ and $|x^*(x)|\geq\tfrac{1}{4}\|x^*\|.$
\end{lemma}

\begin{proof} Suppose that $|x^*(x)|<\tfrac{1}{4}\|x^*\|$ for every $x\geq0$ with $\|x\|\leq1$. In this case, for any $y\in E$ with $\|y\|\leq1$, using that $\|y^+\| \leq \|y\| \leq 1$ and   $\|y^-\| \leq \|y\| \leq 1$, we would have \[|x^*(y)|=|x^*(y^+)-x^*(y^-)|\leq|x^*(y^+)|+|x^*(y^-)|<\frac{1}{4}\|x^*\|+\frac{1}{4}\|x^*\|=\frac{1}{2}\|x^*\|.\] Taking the supremum over $y \in E$ with $\|y\|\leq 1$, we would get %for any $y\in E$ with $\|y\|\leq1$. Tomando o supremo sobre tais $y$ obteríamos
$0 < \|x^*\|\leq\tfrac{1}{2}\|x^*\|$. %Essa contradição completa a demonstração.
\end{proof}

Next we establish necessary and sufficient conditions on a Banach lattice $F$ so that the adjoint of every bounded operator from any Banach space to $F$ is $DW$-compact.

\begin{theorem}\label{teo.2} The following are equivalent for a Banach lattice $F$.
\vspace*{-0.4em}
\begin{enumerate}
\item[\rm(i)] $F^*$ is discrete and has the positive Schur property (or, equivalently, the Schur property).

\vspace*{-0.4em}
\item[\rm(ii)] The adjoint of every bounded operator from any Banach space to $F$ is $DW$-compact.

 \vspace*{-0.4em}
\item[\rm(iii)] The adjoint of every bounded operator from any Banach lattice to $F$ is $DW$-compact.

\vspace*{-0.4em}
\item[\rm(iv)] The adjoint of every bounded operator from any Banach lattice $E$ such that the norm of $E^*$ is not order continuous to $F$ is $DW$-compact. %Para todo reticulado de Banach $E$ tal que $E^*$ não tem norma ordem contínua vale que: Todo operador contínuo $T\colon E\to F$ tem adjunto $T^*\colon F^*\to E^*$ que é $DW$-compacto.
%Lembrar de modificar demonstração do corol 2.12 caso mudar e deste teorema.

 \vspace*{-0.4em}
\item[\rm(v)] The adjoint of every regular operator from any Banach lattice $E$ such that the norm of $E^*$ is not order continuous to $F$ is $DW$-compact.%Para todo reticulado de Banach $E$ tal que $E^*$ não tem norma ordem contínua vale que: Todo operador regular $T\colon E\to F$ tem adjunto $T^*\colon F^*\to E^*$ que é $DW$-compacto.

\vspace*{-0.4em}
\item[\rm(vi)] The adjoint of every positive operator from any Banach lattice $E$ such that the norm of $E^*$ is not order continuous to $F$ is $DW$-compact. %Para todo reticulado de Banach $E$ tal que $E^*$ não tem norma ordem contínua vale que: Todo operador positivo $T\colon E\to F$ tem adjunto $T^*\colon F^*\to E^*$ que é $DW$-compacto.

\vspace*{-0.4em}
\item[\rm(vii)] The adjoint of every positive $DW$-compact operator from $\ell_1$ to $F$ is $DW$-compact. %Todo operador positivo $DW$-compacto $T\colon \ell_1\to F$ tem adjunto $T^*\colon F^*\to \ell_\infty$ que é $DW$-compacto.

\vspace*{-0.4em}
\item[\rm(viii)] There exists a Banach lattice $E$ so that the norm of $E^*$ is not order continuous and the adjoint of every positive $DW$-compact operator from $E$ to $F$ is $DW$-compact. %Existe um reticulado de Banach $E$ tal que $E^*$ não tem norma ordem contínua e vale que: Todo operador positivo $DW$-compacto $T\colon E\to F$ tem adjunto $T^*\colon F^*\to E^*$ que é $DW$-compacto.
\end{enumerate}
\end{theorem}

\begin{proof} {\rm(i)} $\!\Longrightarrow$ \!{\rm(ii)} For the coincidence of the positive Schur property and the Schur property on discrete Banach lattices, see \cite[Remark p.\,171]{Wnuk2}. Let $X$ be a Banach space and let $T \colon X \longrightarrow F$ be a bounded operator. As mentioned before, the norm of every Banach lattice having the positive Schur property is order  continuous. Thus, $F^*$ is a discrete Banach lattice with order continuous norm. From \cite[Theorem 6.1]{Wnuk99} it follows that all order intervals of $F^*$ are norm compact. Therefore, every continuous operator from $F^*$ to any Banach space is $AM$-compact, in particular, $T^*$ is $AM$-compact. Since $F^*$ has the positive Schur property and $T^*$ is $AM$-compact, $T^*$ is Dunford-Pettis by \cite[Theorem 3.2]{Moussa}, hence $T^*$ is $DW$-compact.

The implications {\rm(ii)} $\!\Longrightarrow$ \!{\rm(iii)} $\!\Longrightarrow$ \!{\rm(iv)} $\!\Longrightarrow$ \!{\rm(v)} $\!\Longrightarrow$ \!{\rm(vi)} are straightforward, and {\rm(vii)} $\!\Longrightarrow$ \!{\rm(viii)} holds because the norm of $\ell_1^* = \ell_\infty$ is not order continuous.
%Para reticulados de Banach discretos a propriedade de Schur coincide com a propriedade de Schur positiva (veja \cite[Remark p. 171]{Wnuk2}).
%As implicações $(ii)\Rightarrow(iii)\Rightarrow(iv)\Rightarrow(v)\Rightarrow(vi)\Rightarrow(vii)\Rightarrow(viii)$ são óbvias.
%Vejamos que $(viii)\Rightarrow(i)$:

{\rm(viii)} $\!\Longrightarrow$ \!{\rm(i)} Fix a Banach lattice $E$ so that the norm of $E^*$ is not order continous. By \cite[Theorem 2.4.2]{MeyerPeter}(iv) there exist a positive disjoint sequence $(x^*_n)_n$ in $E^*$ and a functional $x^*\in E^*$ such that $\|x^*_n\|=1$ and $0\leq x^*_n\leq x^*$ for every $n\in \N$. For all $m\in\N$ and $x\in E$, \[\sum_{n=1}^m|x^*_n(x)|\leq\sum_{n=1}^mx^*_n(|x|)=\left(\bigvee_{n=1}^m x^*_n\right)(|x|)\leq x^*(|x|)<\infty,\]
therefore the map
 \[S\colon E\longrightarrow\ell_1~,~ S(x)=(x^*_n(x))_n,\]
  is a well defined positive operator, in particular $S$ is order bounded. The Schur property of $\ell_1$ \cite[Theorem 6.2.12]{livro} guarantees that the operator $S$ is Dunford-Pettis; and it is $AM$-compact because it is order bounded and order intervals in $\ell_1$ are norm compact. % (see \cite[Proposition 1]{Aqzzouz}?????
  Thus, $S$ is $DW$-compact.

Suppose that $F^*$ fails the positive Schur property. %, suponhamos, por contradição, que isso não seja verdade. Nesse caso,
In this case, there exists a normalized positive disjoint weakly null sequence $(y^*_n)_n$ in $F^*$, see, e.g., \cite[Proposition 2.1]{Aqzzouz3}. % com $\|y^*_n\|=1$ para todo $n\in\N$.
Since each $y^*_n$ is a positive linear functional, we have %é um funcional linear positivo, temos
$$1 = \|y^*_n\| = \sup\{|y^*_n(y)| : y\geq 0, \|y\| \leq 1\}$$
(see \cite[Proposition 1.3.5]{MeyerPeter}). %Sabemos que a norma de um operador positivo é dada por vetores positivos da bola unitária,
So, for each $n$ we can pick a positive vector $y_n\in F$ such that $\|y_n\|=1$ and $y^*_n(y_n)\geq\frac{1}{2}$. It is clear that the map %Consideramos o operador
\[R\colon\ell_1\longrightarrow F~,~ R((a_n)_n)=\sum_{n=1}^\infty a_ny_n,\] is a well defined positive operator, hence continuous. $R$ is a Dunford-Pettis operator because $\ell_1$ has the Schur property. Since order  intervals in  $\ell_1$ are norm compact and $R$ is continuous, $R$ is $AM$-compact; therefore $R$ is $DW$-compact. Defining $T:=R\circ S\colon E\longrightarrow F$, we have $T(x)=\sum\limits_{n=1}^\infty x^*_n(x)y_n$ for every $x\in E$. Note that, for being the composition of $AM$-compact operators, $T$ is $AM$-compact as well. Using the weak-to-weak continuity of $S$, the Schur property of $\ell_1$ and the continuity of $F$ it follows that $T$ is a Dunford-Pettis operator. % If Agora, se $(v_n)_n$ é uma sequência fracamente nula em $E$, então $S(v_n)\to 0$ fracamente em $\ell_1$, e portanto  converge para zero em norma, pois $\ell_1$ tem a propriedade de Schur. Daí,  $\|T(v_n)\|=\|R(S(v_n))\|\to 0$, o que mostra que $T$ é Dunford-Pettis.
Therefore, $T$ is $DW$-compact. % pelo Teorema \ref{teo.2.5.dw}. Vejamos que seu
Let us see that its adjoint \[T^*\colon F^*\longrightarrow E^*~,~  T^*(w^*)=\sum_{n=1}^\infty w^*(y_n)x^*_n,\] is not almost Dunford-Pettis. Indeed, for each $k\in\N$,
\[T^*(y^*_k)=\sum_{n=1}^\infty y^*_k(y_n)x^*_n\geq y^*_k(y_k)x^*_k\geq 0,\] from which it follows that %e usando que a norma reticulada preserva a desigualdade temos
$$\|T^*(y^*_k)\|\geq \|y^*_k(y_k)x^*_k\|=y^*_k(y_k)\geq\frac{1}{2} \mbox{ for every } k \in \mathbb{N}.$$ As $(y_k^*)_k$ is a disjoint positive weakly null sequence in $F^*$, we have that $T^*$ is not almost Dunford-Pettis, hence $T^*$ is not $DW$-compact. This contradicts (viii), so $F^*$ has the positive Schur property.

Suppose now that $F^*$ is not discrete. By \cite[Theorem 3.1]{Chen} there exist a weak$^*$-null sequence $(z^*_n)_n$ in $F^*$ and $z^*\in F^*$ so that $|z^*_n|=z^*>0$ for every $n$. For each $n$ Lemma \ref{lemafim} gives a vector $z_n\in F^+$ such that  $\|z_n\|\leq1$ and $|z^*_n(z_n)|\geq\tfrac{1}{4}\|z^*_n\|=\tfrac{1}{4}\|z^*\|$. It is plain that \[U\colon \ell_1\longrightarrow F~,~ U((a_n)_n)=\sum_{n=1}^\infty a_nz_n,\] is a well defined positive operator. Moreover, $U$ is $AM$-compact because it is continuous and order intervals in $\ell_1$ are compact; and $U$ is Dunford-Pettis due to the Schur property of $\ell_1$; hence %. Como $\ell_1$ tem a propriedade de Schur, o operador $U$ é Dunford-Pettis, e portanto
$U$ is $DW$-compact As we did  before, the positive operator \[T:=U\circ S\colon E\longrightarrow F~,~ T(x)=\sum_{n=1}^\infty x^*_n(x)z_n,\] is $DW$-compact. Let us see that its adjoint
\[T^*\colon F^*\longrightarrow E^*~,~ T^*(w^*)=\sum_{n=1}^\infty w^*(z_n)x^*_n,\] fails to be $AM$-compact. Since $(z^*_n)_n$ is weak$^*$-null in $F^*$ and adjoint operators are  weak$^*$-weak$^*$ continuous, the sequence $(T^*(z^*_n))_n$ is weak$^*$-null in $E^*$. Using that $(x^*_n)_n$ is a positive disjoint sequence in  $E^*$ and the continuity of the lattice operations, we get
\begin{align*}|T^*(w^*)|&=\left|\sum_{n=1}^\infty w^*(z_n)x^*_n\right|= \left|\lim_m\sum_{n=1}^m w^*(z_n)x^*_n\right| = \lim_m\left|\sum_{n=1}^m w^*(z_n)x^*_n\right|\\&
= \lim_m \sum_{n=1}^m |w^*(z_n)|\,x^*_n =\sum_{n=1}^\infty |w^*(z_n)|\,x^*_n %=\bigvee\limits_{n=1}^\infty|w^*(z_n)|\,x^*_n
\geq|w^*(z_k)|\,x^*_k
\end{align*}
for all $w^*\in F^*$ and $k\in\N$. Hence, for every $k$, \[\|T^*(z^*_k)\|\geq\||z^*_k(z_k)|x_k\|=|z^*_k(z_k)|\geq\tfrac{1}{4}\|z\|>0.\]
In particular, the sequence $(T^*(z^*_n))_n$ has no convergent subsequence in $E^*$. So, the sequence $(z_k^*)_k$ is order bounded (recall that $|z_k^*|= z^*$ for every $k$) and $(T^*(z^*_n))_n$ is not relatively compact. Therefore, $T^*$ is not  $AM$-compact, in particular, $T$ is $DW$-compact and $T^*$ is not. This contadicts (viii), so $F^*$ is discrete.
%Como caso particular do que fizemos acima, temos que $(v)\Rightarrow(i)$ pois $(\ell_1)^*=\ell_\infty$ não tem norma ordem contínua.
%Resta mostrar $(i)\Rightarrow(ii)$. Com efeito, a norma de $F^*$ é ordem contínua porque $F^*$ tem a propriedade de Schur positiva. Assim, $F^*$ é discreto e tem norma ordem contínua, logo, de \cite[Theorem 6.1]{Wnuk}, segue que os ordem intervalos de $F^*$ são compactos em norma. Dessa forma, todo operador contínuo de $F^*$ em $X^*$ é $AM$-compacto, de modo que $T^*$ é $AM$-compacto. Como $F^*$ tem a propriedade de Schur positiva e $T^*$ é $AM$-compacto, então $T^*$ é Dunford-Pettis pelo \cite[Theorem 3.2]{Moussa}. Concluindo o que queríamos.
\end{proof}

%\textcolor{blue}{Eu acho melhor tirar a aplicação que vc colocou como Exemplo 3.5. Vc diz que todo operador contínuo $T \colon \ell_1 \longrightarrow c_0$ tem adjunto $T^*\colon \ell_1 \longrightarrow \ell_\infty$ que é $DW$-compacto. Veja que isso é trivial: como $\ell_1$ tem a propriedade de Schur, $T^*$ é Dunford-Pettis; e como ordem intervalos em $\ell_1$ são relativamente compactos e $T^*$ é contínuo, $T^*$ é $AM$-compacto, portanto $T^*$ é $DW$-compacto. Ou seja, todo operador contínuo que tem $\ell_1$ no domínio é $DW$-compacto. E portanto o adjunto de todo operador tomando valores em $c_0$ é $DW$-compacto. }

The implication of the adjoint problem we are currently working with was treated for positive Dunford-Pettis operators in \cite{Aqzzouz5} and for regular $AM$-compact operators in \cite{Aqzzouz4}. For the class of $DW$-compact operators there is no need to impose any condition on the operators: Next we give a necessary and sufficient condition so that the implication holds for every operator:

\begin{corollary} The following are equivalent for the Banach lattices $E$ and $F$.
%\textcolor{blue}{\sout{The following are equivalent for a Banach lattice $F$.}}
\vspace*{-0.4em}
\begin{enumerate}
\item[\rm(i)] The adjoint of every $DW$-compact operator from $E$ to $F$ is $DW$-compact.

 \vspace*{-0.4em}
\item[\rm(ii)] The adjoint of every positive $DW$-compact operator from $E$ to $F$ is $DW$-compact.
    \vspace*{-0.4em}
\item[\rm(iii)] $E^*$ has order continuous norm or $F^*$ is discrete and has the positive Schur property (or, equivalently, the Schur property).
\end{enumerate}
\end{corollary}

\begin{proof} {\rm(ii)} $\!\Longrightarrow$ \!{\rm(iii)} If the norm of $E^*$ is not order continuous, the assumption and the proof of Theorem  \ref{teo.2}{\rm(viii)} $\!\Longrightarrow$ \!{\rm(i)} yield that $F^*$ is discrete and has the positive Schur property.

{\rm(iii)} $\!\Longrightarrow$ \!{\rm(i)} If the norm of $E^*$ is order continuous, then every $DW$-compact operator $T \colon E \longrightarrow F$ is compact by \cite[Corollary 2.2]{Jin}, therefore $T^*$ is compact, hence $DW$-compact, by Schauder's Theorem. If $F^*$ is discrete and has the positive Schur property, then the result follows from Theorem \ref{teo.2}.
\end{proof}

\section{$DW$-$DP$ operators}

The purpose of this section is just to start the duality theory of the class of $DW$-$DP$ operators, actually we shall prove only one result in this direction.

Recall that a bounded subset $A$ of a Banach space $X$ is a {\it Dunford-Pettis} set if, regardless of the Banach space $Y$ and the weakly compact operator $u \colon X \longrightarrow Y$, $u(A)$ is relatively compact in $Y$.

According to \cite{Jin2}, an operator from a Banach lattice to a Banach space is a {\it $DW$-$DP$ operator} if it maps disjoint weakly compact sets to Dunford-Pettis sets.

It is easy to see that $DW$-compact operators are $DW$-$DP$. Again, a useful result characterizes the class of $DW$-$DP$ operators as the intersection of two previously studied classes. Let $E$ be a Banach lattice and let $X,Y$ be Banach spaces. \\
$\bullet$ An operator $T \colon E \longrightarrow Y$ is {\it order Dunford-Pettis} if $T$ maps order bounded subsets of $E$ to Dunford-Pettis subsets of $Y$ (see \cite{Aqzzouz9, Bouras}). \\
$\bullet$ An operator $T \colon X \longrightarrow Y$ is {\it weak Dunford-Pettis} if $y_n^*(T(x_n)) \longrightarrow 0$ whenever $(x_n)_n$ and $(y_n^*)_n$  are weakly null sequences in $X$ and $Y^*$, respectively. Or, equivalently, if $T$ sends relatively weakly compact subsets of $X$ to Dunford-Pettis subsets of $Y$ (see, \cite[p.\,349]{Ali1}).

\begin{theorem}\label{nnth} {\rm \cite[Theorem 2.12]{Jin2}} An operator from a Banach lattice to a Banach space is a $DW$-$DP$ operator if and only if it is order Dunford-Pettis and weak Dunford-Pettis.
\end{theorem}

We shall give a sufficient and a necessary condition on the underlying spaces so that one of implications of the adjoint problem for $DW$-$DP$ operators holds. To do so, we need the following lemma.

\begin{lemma}\label{nlemn} If there exists a positive Banach space isomorphism from the Banach lattice $G$ to $\ell_\infty$, then ${\rm id}_{G^*}$ is a $DW$-$DP$ operator.
\end{lemma}

\begin{proof} As $\ell_\infty^*$ is an $AL$-space, its norm is order continuous \cite[p.\,194]{Ali1} and it has the Dunford-Pettis property \cite[Theorem II.9.9]{Schaefer}, that is, every weakly compact operator from $\ell_\infty^*$ to any Banach space is Dunford-Pettis. So, relatively weakly compact subsets of $\ell_\infty^*$, in particular order intervals, are Dunford-Pettis. This proves that ${\rm id}_{\ell_\infty^*}$ is a $DW$-$DP$ operator by Theorem \ref{nnth}.

Let $S \colon \ell_\infty \longrightarrow G$ be a positive Banach space isomorphism. %, that is, a Banach space isomorphism that is also a Riesz homomorphism.
Since $S$ is positive, $S^*$ is positive as well. By \cite[Proposition 4.3.11]{livro}, $S^* \colon G^* \longrightarrow \ell_\infty^*$ is a Banach space isomorphism. %, and now it is easy to see that $(S^*)^{-1}$ is also a Riesz homomorphism \textcolor{blue}{(veja abaixo)}.
So, $S^*$ is a positive operator, hence order bounded, $(S^*)^{-1}$ is a bounded operator,  ${\rm id}_{G^*} = (S^*)^{-1} \circ {\rm id}_{\ell_\infty^*} \circ S^*$ and ${\rm id}_{\ell_\infty^*}$ is order Dunford-Pettis and weak Dunford-Pettis. By \cite[Proposition 3.1]{Bouras}, ${\rm id}_{G^*}$ is order Dunford-Pettis and by \cite[Exercise 5.14, pg.\,356]{Ali1} it is weak Dunford-Pettis, therefore it is $DW$-$DP$.
%\textcolor{blue}{Comecemos provando que a inversa de um homomorfismo de Riesz bijetor $T$ também é um homomorfismo de Riesz:
%$$|x| = |x| \Longrightarrow |T(T^{-1}(x)| = |x| \Longrightarrow T(|T^{-1}(x)|) = |x| \Longrightarrow T^{-1}(|x|) = |T^{-1}(x)|.  $$Vejamos que se $u$ é isomorfismo de Riesz, então $(u^*)^{-1}$ também é isomorfismo de Riesz. Por \cite[Proposition 4.3.11]{livro} sabemos que $u^{-1}$ é isomorfismo de espaços de Banach, e pelo que fizemos acima, também é isomorfismo de Riesz. Como $u$ é positivo, segue facilmente que $u^*$ também é positivo. Assim, $u^*$ é um isomorfismo de Riesz   }
\end{proof}

\begin{theorem} \label{Teo.DWDP}
Consider the following assertions about the $\sigma$-Dedekind complete Banach lattices $E$ and $F$.
\vspace*{-0.5em}
\begin{enumerate}
\item[\rm(i)] $E$ has order continuous norm.
\vspace*{-0.4em}
\item[\rm (ii)] An operator $T\colon E\longrightarrow F$ is $DW$-$DP$ whenever  $T^*\colon F^*\longrightarrow E^*$ is weak Dunford-Pettis.
\vspace*{-0.4em}
\item[\rm (iii)] An operator $T\colon E\longrightarrow F$ is $DW$-$DP$ whenever  $T^*\colon F^*\longrightarrow E^*$ is $DW$-$DP$.
\vspace*{-0.4em}
\item[\rm (iv)] $E$ or $F$ has order continuous norm.
\end{enumerate}
\vspace*{-0.4em}
Then {\rm(i)} $\!\Longrightarrow$ \!{\rm(ii)} $\!\Longrightarrow$ {\!\rm(iii)} $\!\Longrightarrow$ \!{\rm(iv)}.
%Então $(i)\Rightarrow(ii)\Rightarrow(iii)\Rightarrow(iv)$.
\end{theorem}

\begin{proof} {\rm(i)} $\!\Longrightarrow$ \!{\rm(ii)} Let $T\colon E\longrightarrow F$ be an operator whose adjoint $T^*\colon F^*\longrightarrow E^*$ is weak Dunford-Pettis. It follows from \cite[Theorem 3.1]{Aqzzouz8} that $T$ is weak  Dunford-Pettis. As $E$ has order continuous norm, the interval $[-x,x]$ is weakly compact in $E$ for every $x\in E_+$ \cite[Theorem 2.4.2]{MeyerPeter}. Then $T[-x,x]$ is Dunford-Pettis for every $x\in E_+$, because $T$ is weak Dunford-Pettis. Therefore $T$ is order Dunford-Pettis, hence it is $DW$-$DP$.

{\rm(ii)} $\!\Longrightarrow$ \!{\rm(iii)} This implication follows from Theorem \ref{nnth}.

{\rm(iii)} $\!\Longrightarrow$ \!{\rm(iv)} Suppose that (iv) does not hold, that is, the norms of $E$ and $F$ are not order continuous. By \cite[Theorem 4.51]{Ali1} there are lattice embeddings $S_1\colon \ell_\infty\longrightarrow E$ and $S_2\colon \ell_\infty\longrightarrow F$. Considering the closed sublattices $G_1:= S_1(\ell_\infty)$ of $E$ and $G_2 := S_2(\ell_\infty)$ of $F$, we can take the  Banach lattice isomorphisms $(S_1)^{-1}\colon G_1\longrightarrow \ell_\infty$ and $(S_2)^{-1}\colon G_2\longrightarrow\ell_\infty$. From \cite[Corollary 3, pg.\,111]{Schaefer}, there are positive operators $U_1\colon E\longrightarrow\ell_\infty$ and $U_2\colon F\longrightarrow\ell_\infty$ that extend $(S_1)^{-1}$ and $(S_2)^{-1}$, respectively. Thus $P_1:=S_1\circ U_1\colon E\longrightarrow E$ is a positive projection onto $G_1$ and $P_2:=S_2\circ U_2\colon F\longrightarrow F$ is a positive projection onto $G_2$. Note that
  $$i := S_2\circ (S_1)^{-1}\colon G_1\stackrel{(S_1)^{-1}}{\longrightarrow}\ell_\infty\stackrel{S_2}{\longrightarrow}F$$ is a lattice embedding, in particular it is a positive operator. Define
  $$T:=i\circ P_1\colon E\stackrel{P_1}{\longrightarrow}G_1\stackrel{i}{\longrightarrow}F,$$ and note that $$T^*=P_1^*\circ i^*=P_1^*\circ {\rm id}_{G_1^*}\circ i^*\colon F^*\stackrel{i^*}{\longrightarrow}G_1^*\stackrel{{\rm id}_{G_1^*}}{\longrightarrow}G_1^*\stackrel{P_1^*}{\longrightarrow}E^*\,.$$
  Since ${\rm id}_{G_1^*}$ is a  $DW$-$DP$ operator by Lemma \ref{nlemn}, $i^*$ is a positive operator as the adjoint of a lattice embedding, and $P_1^*$ is a continuous operator, $T^*$ is order Dunford-Pettis by \cite[Proposition 3.1]{Bouras} and weak Dunford-Pettis by \cite[Exercise 5.14, pg.\,356]{Ali1}, therefore $T^*$ is $DW$-$DP$.

Suppose that $T$ is a $DW$-$DP$ operator, in particular $T$ is order Dunford-Pettis. Since $P_2$ is a positive operator, %and $P_2$ is continuous,  Caso contrário, o operador
$$R:= P_2 \circ T = P_2\circ i\circ P_1\colon E\stackrel{P_1}{\longrightarrow}G_1\stackrel{i}{\longrightarrow}F\stackrel{P_2}{\longrightarrow}G_2$$ is order Dunford-Pettis by \cite[Proposition 3.1]{Bouras}.
%For the same reason, its restriction to $G_1$:
%$$R := P_2\circ i\circ P_1\vert_{G_1}\colon G_1\longrightarrow G_2,$$ which is a Riesz isomorphism \textcolor{blue}{(por quê isso é verdade e por quê isso está aqui ?????)}, is order Dunford-Pettis as well. Applying the same reason for the third time,
For the same reason, it follows that $$(S_2)^{-1}\circ R\circ S_1=(S_2)^{-1}\circ P_2\circ S_2\circ (S_1)^{-1}\circ P_1\circ S_1={\rm id}_{\ell_\infty}$$
is order Dunford-Pettis. But this is not true: the order interval $[-(1,1,1,\ldots),(1,1,1,\ldots)]$ is not a Dunford-Pettis subset of $\ell_\infty$ (see \cite[Remark 3]{Bouras}). This proves that $T$ is not a $DW$-$DP$ operator, hence (iii) does not hold. The proof is complete.
\end{proof}

The last part of the proof above actually proves the following:

\begin{corollary} Let $E$ and $F$ be $\sigma$-Dedekind complete Banach lattices. If every operator from $E$ to $F$ having order Dunford-Pettis adjoint is order Dunford-Pettis, then $E$ and $F$ have order continuous norms. %adjoint, Um operador $T\colon E\to F$ contínuo é ordem Dunford-Pettis sempre que seu adjunto $T^*\colon F^*\to E^*$ é ordem Dunford-Pettis, então $E$ ou $F$ têm normas ordem contínuas.
\end{corollary}

\bigskip
\noindent Geraldo Botelho~~~~~~~~~~~~~~~~~~~~~~~~~~~~~~~~~~~~~~Ariel Mon\c c\~ao\\
Instituto de Matem\'atica e Estat\'istica~~~~~~~~~\,\,\,Departamento de Matemática\\
Universidade Federal de Uberl\^andia~~~~~~~~~~~~~Universidade Federal de Minas Gerais\\
38.400-902 -- Uberl\^andia -- Brazil~~~~~~~~~~~~~~~~~31.270-901 -- Belo Horizonte -- Brazil\\
e-mail: botelho@ufu.br~~~~~~~~~~~~~~~~~~~~~~~~~\,~~~~~e-mail: arieldeom@hotmail.com
%
%
%
%
%%\noindent Mikaela Aires~~~~~~~~~~~~~~~~~~~~~~~~~~~~~~~~~~~~~~~~~~~~~~Geraldo Botelho~\\
%%Instituto de Matem\'atica e Estat\'istica~~~~~~~~~~~~~~~Instituto de  Matem\'atica e Estat\'istica\\
%%Universidade de S\~ao Paulo~~~~~~~~~~~~~~~~~~~~~~~~~~~~~\hspace*{0,1em}Universidade Federal de Uberl\^andia\\
%%05.508-090 -- S\~ao Paulo -- Brazil~~~~~~~~~~~~~~~~~~~~~~\,\hspace*{0,1em}38.400-902 -- Uberl\^andia -- Brazil\\
%%e-mail: mikaela\_aires@ime.usp.br~~~~~~~~~~~~~~~~~~~~~\,e-mail: botelho@ufu.br
%%\bigskip
%%
%%

\end{document}